\documentclass[12pt]{amsart}
\usepackage[margin=1in]{geometry}                
\usepackage{graphicx}
\usepackage{amssymb}
\usepackage{amsthm}
\usepackage{mathrsfs}
\usepackage{caption}
\usepackage{xcolor}
\usepackage{verbatim}

\newtheorem{theorem}{Theorem}[section]
\newtheorem{proposition}[theorem]{Proposition}
\newtheorem{lemma}[theorem]{Lemma}
\newtheorem{definition}[theorem]{Definition}
\newtheorem{corollary}[theorem]{Corollary}

\newtheorem{remark}{Remark}[section]

\numberwithin{equation}{section}

\newcommand{\ZZ}{\mathbb{Z}}
\newcommand{\Z}{\mathbb{Z}}
\newcommand{\RR}{\mathbb{R}}
\newcommand{\NN}{\mathbb{N}}
\newcommand{\R}{\mathbb{R}}

\newcommand{\calm}{\mathcal{M}}
\newcommand{\calL}{\mathcal{L}}

\newcommand{\Div}{{\mathrm{Div}}}

\title{Keller properties for integer tilings} 
\author{Benjamin Bruce and Izabella {\L}aba}

\date{\today}

\begin{document}

\begin{abstract}
Keller's conjecture on cube tilings asserted that, in any tiling of $\RR^d$ by unit cubes, there must exist two cubes that share a $(d-1)$-dimensional face. This is now known to be true in dimensions $d\leq 7$ and false for $d\geq 8$. In this article, we investigate  analogues of Keller's conjecture for integer tilings.

\end{abstract}

\maketitle



\section{Introduction}


\subsection{Motivation}
Let $A\subset\ZZ$ be a finite and nonempty set. 
We say that $A$ {\em tiles the integers by translations} if 
there exists a translation set $T\subset\ZZ$ such that every integer $n\in\ZZ$ can be written uniquely as $n=a+t$ with $a\in A$ and $t\in T$. Informally, $\ZZ$ can be covered by pairwise disjoint translates of $A$. We will refer to such $A$ as an {\it integer tile}. 

It is well known \cite{New} that any tiling of $\ZZ$ by a finite set $A$ must be periodic, so that there exists an $M\in\NN$ and a finite set $B\subset \ZZ$ such that $T=B\oplus M\ZZ$. Thus $A\oplus B\oplus M\ZZ=\ZZ$; in other words, $A\oplus B$ mod $M$ is a factorization of the cyclic group $\ZZ_M$. We will write this as $A\oplus B=\ZZ_M$. Since translating an element of $A$ or $B$ by a multiple of $M$ does not change that property, we will consider $A,B$ as subsets of $\ZZ_M$.

The tiling condition $A\oplus B=\ZZ_M$
has an equivalent formulation in terms of the {\em divisor sets} of $A$ and $B$. 
Fix $M\in\NN$; given two integers $m,n\in\ZZ$, we will use $(m,n)$ to denote their greatest common divisor. For a finite set $A\subset\ZZ$, we define the divisor set of $A$ relative to $M$:
\begin{equation}\label{def-div}
\Div(A)=\{(a-a',M):\ a,a'\in A\}.
\end{equation}
Informally, we refer to the elements of $\Div(A)$ as the {\em divisors} of $A$.

\begin{theorem}\label{thm-sands} {\bf(Sands)}
Let $A,B\subset\ZZ_M$ be sets. Then  $ A\oplus B=\ZZ_M$ is a tiling if and only if
\begin{equation}\label{e-sands}
|A|\,| B|=M\hbox{ and }\Div( A) \cap \Div( B) =\{M\}.
\end{equation}
\end{theorem}

Thus, if $A\oplus B=\ZZ_M$, (\ref{e-sands}) says that each $m|M$ with $m\neq M$ may belong to at most one of $\Div(A)$ and $\Div(B)$. It does not, however, specify which divisors must actually occur in $\Div(A)\cup\Div(B)$, nor does it say how they might be distributed between these two sets.

In this paper, we investigate the following question: must every tiling $A\oplus B=\ZZ_M$ satisfy
\begin{equation}\label{Mp-question}
M/p \in \Div(A)\cup \Div (B)\hbox{ for some prime }p|M?
\end{equation}
 While we are not able to resolve the problem in its full generality, we do have both positive and negative results in this direction. 

Our interest in the above question is motivated by several considerations. One is to advance the understanding of the general structure of integer tilings. In this regard, the main open problem is the Coven-Meyerowitz conjecture \cite{CM}, which we now describe briefly (see Section \ref{sec-CM} for more details). For finite sets $A,B\subset\NN\cup\{0\}$, we define their {\it mask polynomials}
$$
A(X)=\sum_{a\in A}X^a,\ B(X)=\sum_{b\in B}X^b .
$$
Then $A\oplus B=\ZZ_M$ is equivalent to
\begin{equation}\label{mask-e1}
 A(X)B(X)=1+X+\dots+X^{M-1}\ \mod (X^M-1).
\end{equation}
This can be rephrased further in terms of the cyclotomic divisors of $A(X)$ and $B(X)$.
Recall that the $s$-th cyclotomic
polynomial $\Phi_s(X)$ is the unique monic, irreducible polynomial whose roots are the
primitive $s$-th roots of unity. Equivalently, $\Phi_s$ can be computed from
the identity $X^n-1=\prod_{s|n}\Phi_s(X)$.
Then we may rewrite (\ref{mask-e1}) as
\begin{equation}\label{mask-e2}
  |A||B|=M\hbox{ and }\Phi_s(X)\, |\, A(X)B(X)\hbox{ for all }s|M,\ s\neq 1.
\end{equation}
Since $\Phi_s$ are irreducible, each $\Phi_s(X)$ with $s|M$ and $s\neq 1$ must divide at least one of $A(X)$ and $B(X)$.
Coven and Meyerowitz \cite{CM} proposed conditions on how these cyclotomic divisors may be distributed between $A(X)$ and $B(X)$; the statement that all integer tiles must satisfy these conditions has become known as the Coven-Meyerowitz conjecture. The conjecture allows an equivalent statement in terms of the divisor sets defined in (\ref{def-div}) \cite[Proposition 3.4]{LL1}. It has been proved in certain significant special cases \cite{CM, LL2, LL3}, but remains open in general and appears to be very difficult to resolve.

The recent papers \cite{LL1,LL2,LL3} established a link between the Coven-Meyerowitz conjecture and structural properties such as (\ref{Mp-question}). Specifically, the approach in \cite{LL2, LL3} depends on being able to find certain ``fibered" structures in one or both of the tiles. This, in particular, requires a stronger version of (\ref{Mp-question}) to hold. Our main negative result, Theorem \ref{thm1}, provides examples of integer tilings where such structures do not appear in either tile. This limits the applicability of some of the methods of \cite{LL2,LL3} to tilings where $M$ has a small number of distinct prime factors. 
Conversely, if properties such as (\ref{Mp-question}) can be established, they could provide partial structural information in cases where the full conjecture is not available. We discuss this in more detail in Section \ref{sec-CM}.

More generally, one can ask ``how complicated can integer tilings really get?''. 
Given the existing significant body of work on high-dimensional tilings with counterintuitive properties, a natural direction of research is to try to use such examples to construct ``pathological" integer tilings. This connection can provide useful geometrical insights into properties of integer tilings that might otherwise be difficult to visualize. For instance, the examples in \cite{LS}, \cite{Sz} (disproving a conjecture of Sands \cite{Sands-Hajos} on factorization of finite abelian groups) have a natural interpretation in terms of 3-dimensional cube tilings with ``shifted columns''. In \cite{LL1,LL2,LL3}, this geometrical interpretation played a significant role in the classification of tilings of $\ZZ_{p^2q^2r^2}$, where $p,q,r$ are distinct primes. In \cite{Kol, Steinberger2}, it was used to construct integer tilings with long periods.

Keller's conjecture on cube tilings (see Section \ref{sec-Kellerforcubes} for the relevant background) stated that in any tiling of $\RR^d$ by translates of the unit cube, there must be two cubes that share a full $(d-1)$-dimensional face.
The counterexamples found in \cite{LShor1, Mackey} provide an important class of counterintuitive tilings in high dimensions. It is natural to ask whether such phenomena have meaningful counterparts in the setting of integer tilings. 
Our question (\ref{Mp-question}) is a natural analogue of Keller's face-sharing property. 
However, it turns out that the counterexamples in \cite{LShor1, Mackey} are not easily adapted; we are only able to use them to disprove a stronger property than (\ref{Mp-question}). 

The tilings we construct are known to satisfy the Coven-Meyerowitz tiling conditions (see Section \ref{sec-CM}). On one hand, this means that such constructions cannot provide a counterexample to the Coven-Meyerowitz conjecture without further significant new ideas. On the other hand, it also shows that integer tilings may have rather complicated structure even when the Coven-Meyerowitz conditions are known to hold.


\subsection{Keller's conjecture for cube tilings}\label{sec-Kellerforcubes}

Let $Q=[0,1)^d$ be the unit cube in $\RR^d$. For the purpose of this paper, a {\it cube tiling} will always mean a tiling of $\RR^d$ by congruent and pairwise disjoint translates of $Q$. Consider the following {\it Keller properties} that a cube tiling $T\oplus Q =\RR^d$ might have. For $i\in\{1,\dots,d\}$, we use $e_i$ to denote the unit vector in the $i$-th direction.

\medskip\noindent
{\bf (KP1)}  There exist $t,t'\in T$ such that $t-t'=e_i$ for some $i\in\{1,\dots,d\}$.

\medskip\noindent
{\bf (KP2)} There exist $t\in T$ and $i\in\{1,\dots,d\}$ such that 
$\{t+ne_i:\ n\in\ZZ\}\subset T$. 

\medskip

The first property states that there are two cubes in the given tiling that share a full $(d-1)$-dimensional face. The second property makes the stronger assertion that the tiling must in fact contain an infinite ``column'' of cubes sharing full $(d-1)$-dimensional faces.

Keller \cite{Keller} conjectured in 1930 that all cube tilings must satisfy (KP2), hence also (KP1). 
The (formally weaker) statement that all cube tilings must satisfy (KP1) has become known in the literature as 
{\it Keller's conjecture}. This is now known to be true in low dimensions but false in general. Perron \cite{Perron1,Perron2} proved that (KP1) holds for all cube tilings of $\RR^d$ with $d\leq 6$. 
The stronger statement that (KP2) holds for all cube tilings in dimensions $d\leq 6$ was proved by {\L}ysakowska and Przes{\l}awski \cite{LP1,LP2}. Brakensiek, Heule, Mackey, and Narv\'aez \cite{BHMN} proved that
(KP1) holds for all unit cube tilings of $\RR^7$. 
In the other direction, 
Lagarias and Shor \cite{LShor1} constructed cube tilings in dimensions $d\geq 10$ that do not have the property (KP1) and, therefore, (KP2). Mackey \cite{Mackey} extended this to dimensions $d\geq 8$.


\subsection{Integer Keller properties}

Assume that $M=\prod_{i=1}^d p_i^{n_i}$,
where $p_1,\dots,p_d$ are distinct primes and $n_1,\dots,n_d\in\NN$. 
Let $A\oplus  B=\ZZ_M$ be a tiling. 
Define the divisor sets $\Div(A)$, $\Div(B)$ as in (\ref{def-div}), and let
 \begin{equation}\label{fiber} 
F_i:= \{0,M/p_i,2M/p_i,\dots,(p_i-1)M/p_i\}\subset\ZZ_M \hbox{ for }i=1,2,\dots,d.
\end{equation}
In \cite{LL1, LL2, LL3}, a translate (coset) of $F_i$ is called an {\it $M$-fiber in the $p_i$ direction.}
Consider the following ``integer Keller properties'' that the tiling might have.

\medskip\noindent
{\bf  (IKP1)} There exists an $i\in\{1,\dots,d\}$ such that $M/p_i\in\Div( A)\cup\Div(B)$.

\medskip\noindent
{\bf (IKP2)} There exist $u\in A$ and $i\in\{1,\dots,d\}$ such that $u+F_i \subset  A$.

\medskip

We will also consider the ``cyclotomic Keller property" below. Unlike (IKP1) and (IKP2), the statement (CKP) concerns sets $A\subset \ZZ_M$ that need not be tiles.

\medskip\noindent
{\bf (CKP)} For every nonempty set $A\subset\ZZ_M$ such that $\Phi_M(X)|A(X)$, there exists an $i\in\{1,\dots,d\}$ such that $M/p_i\in\Div( A)$.

\medskip

Clearly, (IKP2) is a stronger statement than (IKP1). Furthermore, if $A\oplus B=\ZZ_M$ is a tiling, then by (\ref{mask-e2}) $\Phi_M(X)$ divides at least one of $A(X)$ and $B(X)$; hence if (CKP) holds for some $M$, then (IKP1) holds for all tilings of $\ZZ_M$ with that $M$. On the other hand, the failure of (CKP) would not necessarily imply the failure of (IKP1), since there exist nonempty sets $A\subset\ZZ_M$ that satisfy $\Phi_M(X)|A(X)$ but do not tile $\ZZ_M$. We also note that $\Phi_M(X)|A(X)$ does not imply that $u+F_i \subset  A$ for any $u$ or $i$ (hence there is no (CKP2) property).
See e.g.~\cite{PR} for an extensive family of examples.

The geometric interpretation of the above statements is as follows. By the Chinese Remainder Theorem, we have
\begin{equation}\label{CRT}
\ZZ_M=\bigoplus_{i=1}^d \ZZ_{p_i^{n_i}}\, .
\end{equation}
This represents $\ZZ_M$ as a $d$-dimensional lattice, with each cardinal direction corresponding to a prime divisor $p_i$ of $M$, and $p_i^{n_i}$-periodic in each such direction (see Section \ref{subsec-convert} for more details). Then (IKP1) states that at least one of the sets $A,B$ in the given tiling, say $A$, contains two elements $a,a'$ such that $a-a'$ is one of the ``cardinal differences" $M/p_i,2M/p_i,\dots,(p_i-1)M/p_i$ in some direction. The stronger property (IKP2) states that at least one of $A,B$ contains an entire fiber in some direction. Thus (IKP1) and (IKP) can be viewed as the integer counterparts of the properties (KP1) and (KP2) for cube tilings.


\subsection{Results}

The structure of tilings whose period $M$ has at most 3 distinct prime factors is understood well enough to provide the partial results in Theorem \ref{thm-lowdim} below. These results are either stated explicitly in the literature, or else they follow directly from known arguments; we provide the details in Section \ref{sec-proof-thm-lowdim}.

\begin{theorem}\label{thm-lowdim}\cite{deB, LamLeung, LL2, LL3}
Let $M=p_1^{n_1}p_2^{n_2}p_3^{n_3}$ with $n_1,n_2,n_3\in\NN\cup\{0\}$, where $p_1,p_2,p_3$ are distinct primes. Then:

\begin{itemize}

\item [(i)] (CKP) holds for $M$; consequently, (IKP1) holds for any tiling $A\oplus B=\ZZ_M$.

\item [(ii)] Assume further that either $n_3=0$ or $\max(n_1,n_2,n_3)\leq 2$. Then (IKP2) holds for any tiling $A\oplus B=\ZZ_M$.
\end{itemize}

\end{theorem}

Our new result in this regard is the following theorem.

\begin{theorem}\label{cuboidKeller}
Let $M=p_1^{n_1}\dots p_d^{n_d}$, where $p_1<\dots <p_d$ are distinct primes and $n_1,\dots,n_d \in\NN$.
Assume further that 
\begin{equation}\label{largeprimes}
p_j>2^{j-2}\hbox{ for all }j\in\{6,\dots,d\}.
\end{equation}
Then 
(CKP) holds for $M$. Consequently, (IKP1) holds for any tiling $A\oplus B=\ZZ_M$.

\end{theorem}

\begin{remark}\label{cuboid-remark1}
The assumption $p_j>2^{j-2}$ is always true for $j\leq 5$, so that if (\ref{largeprimes}) holds, it actually holds for all $1\leq j\leq d$. In particular, (\ref{largeprimes}) holds when all primes are sufficiently large relative to $d$, so that $p_j>2^{d-2}$ for all $j\in\{1,\dots,d\}$. 
\end{remark}

We also note a special case when one of the primes is very large relative to the others.

\begin{theorem}\label{easyKeller}
Let $M=p_1^{n_1}\dots p_d^{n_d}$, where $p_1,\dots ,p_d\,$ are distinct primes and $n_1,\dots,n_d \in\NN$. Let $A\oplus  B=\ZZ_M$ be a tiling. 
Assume further that 
\begin{equation}\label{pd-large}
p_d>\max\Big( (|A|,M/p_d^{n_d}), (|B|,M/p_d^{n_d})\Big).
\end{equation}
Then (IKP1) holds for this tiling.
\end{theorem}

In many cases, the assumption (\ref{pd-large}) of Theorem \ref{easyKeller} can be weakened; see Theorem \ref{easyKeller2} and Lemma \ref{easyKeller3}.

We do not know of any counterexamples to either (IKP1) or (CKP).
However, there exist integer tilings for which (IKP2) fails. Our counterexamples are provided by the class of tilings we define now.

\begin{definition}\label{def-cubetilings}
Let $M=p_1^2\dots p_d^2$, where $p_1,\dots,p_d$ are distinct primes. Let $M_i=M/p_i^2$ for $i=1,\dots,d$.
An {\em integer cube tiling} is a tiling of the form
$A\oplus B=\ZZ_M$, where
\begin{equation}\label{defineB}
\begin{split}
{B}&=\bigoplus_{j=1}^d \{0,M_j,2M_j,\dots,(p_j-1)M_j\}
\\
&=\left\{\sum_{j=1}^d c_jM_j:\ c_j\in\{0,1,\dots,p_j-1\}, \ j=1,\dots,d\right\}.
\end{split}
\end{equation}
\end{definition}

Equation (\ref{defineB}) implies in particular that $|A|=|B|=p_1\dots p_d$. 
We may view integer cube tilings as tilings of the integer lattice by translates of a discrete rectangular box; see Section \ref{integer-cubes} for more details.

\begin{theorem}\label{thm1}
Assume that $d\geq 8$. For any choice of distinct primes $p_1,\dots,p_d$, there exists an integer cube tiling $ A\oplus B=\ZZ_M$ that does not satisfy (IKP2).
\end{theorem}

Our proof of Theorem \ref{thm1} is an adaptation of the counterexamples to Keller's conjecture in dimensions 8 and higher \cite{LShor1}, \cite{Mackey}, together with a rearrangement argument due to Szab\'o \cite{Sz-Keller}. 
On the other hand, the known positive results towards Keller's conjecture imply the following theorem for this specific type of tiling.

\begin{theorem}\label{lowdim-thm-cube}
All cube integer tilings with $d\leq 6$ satisfy (IKP2) (therefore also (IKP1)).
Moreover, (IKP1) holds for all cube integer tilings with $d=7$.
\end{theorem}

For cube tilings, the property (KP2) is formally stronger than (KP1), but they turn out to be true or false in the same dimensions except possibly for $d=7$ where, to the best of our knowledge, the status of (KP2) is unknown. 
On the other hand, the property (IKP2) is strictly stronger than (IKP1), in the sense that there exist integer cube tilings that satisfy (IKP1) but not (IKP2). To see this, let $d\geq 8$, and let $p_1,\dots, p_d$ be distinct primes. By Theorem \ref{thm1}, for any choice of $p_1,\dots,p_d$ there exists an integer cube tiling $\tilde A\oplus \tilde B=\ZZ_M$ for which (IKP2) does not hold. On the other hand, if $p_1,\dots,p_d$ satisfy the additional assumptions of either 
Theorem \ref{cuboidKeller} or 
Theorem \ref{easyKeller}, then (IKP1) must hold for the same tiling.

The rest of this paper is organized as follows. In Section \ref{subsec-convert}, we detail the conversion between integer tilings and lattice tilings with appropriate periodicity conditions. In Section \ref{cyclo-tools}, we introduce notation and basic cyclotomic divisibility tools to be used in the proofs of Theorems \ref{thm-lowdim}, \ref{cuboidKeller}, and \ref{easyKeller}. We prove these theorems in Sections \ref{sec-proof-thm-lowdim}, \ref{sec-proof-thm-cuboidKeller}, and \ref{sec-proof-thm-easyKeller}.
We prove Theorem \ref{thm1} in Section \ref{thm1-proof}. Finally, in Section \ref{sec-CM}, we discuss the relationship between integer Keller properties and the Coven-Meyerowitz conjecture.


\section{Correspondence between integer tilings and multidimensional lattice tilings}
\label{subsec-convert}

\subsection{The general case}
We establish a natural correspondence between integer tilings and tilings of multidimensional integer lattices satisfying appropriate periodicity conditions. This also provides a correspondence between integer cube tilings and a class of cube tilings of $\RR^d$.

Let $M=\prod_{i=1}^d p_i^{n_i},$
where $p_1,\dots,p_d$ are distinct primes and $n_1,\dots,n_d\in\NN$. It will be convenient to have
a specific coordinate system on $\ZZ_M$ corresponding to (\ref{CRT}). Let $M_i = M/p_i^{n_i} = \prod_{j\neq i} p_j^{n_j}$. Let also $[n]=\{0,1,\dots,n-1\}$ for $n\in\NN$. Define the projection
\begin{equation}\label{def-pi}
\ZZ^d \ni x=(x_1,\dots,x_d) \ \to   \ \pi(x):= \sum_{i=1}^d x_i M_i ,
\end{equation}
and let
\begin{equation}\label{def-L}
\begin{split}
\calL_M& := p_1^{n_1}\ZZ\times \dots\times p_d^{n_d}\ZZ \subset\ZZ^d,
\\
\Lambda_M& := [p_1^{n_1}]\times\dots\times [p_d^{n_d}]\subset\ZZ^d,
\end{split}
\end{equation}
so that $\Lambda_M\oplus\calL_M=\ZZ^d$.

\begin{lemma}\label{pi-lemma}
Let $x\in\ZZ^d$. Then
\begin{equation}\label{LM-equiv}
x\in\calL_M \ \ \Leftrightarrow\  \ \pi(x)\equiv 0 \hbox{ mod }M.
\end{equation}
Furthermore, 
$\pi(\Lambda_M)$ is a complete residue system mod $M$, the projection $\pi$  is one-to-one on $\Lambda_M$, and
$\pi(\calL_M)=M\ZZ$.
%
%
\end{lemma}

\begin{proof}
We first prove (\ref{LM-equiv}). Let
$x\in \calL_M$. Then $M$ divides each term in the sum $\sum_i x_i M_i$, hence $\pi(x)\in M\ZZ$. Conversely, suppose that $M|\pi(x)$, and let $j\in\{1,\dots,d\}$. Then $p_j^{n_j}|\pi(x)= \sum_i x_i M_i$. Since $p_j^{n_j}|M_i$ for all $i\neq j$, and $(p_j,M_j)=1$, we must have $p_j^{n_j}|x_j$. Since this is true for all $j$, we have $x\in \calL_M$. 

Next, if $x,x'\in\Lambda_M$ are distinct, then $x-x'\not\in\calL_M$ by definition. By (\ref{LM-equiv}), we have
$\pi(x)\not\equiv \pi(x')$ mod $M$, implying the statements about $\Lambda_M$.
Clearly, (\ref{LM-equiv}) implies that $\pi(\calL_M)\subset M\ZZ$. The converse inclusion follows from the fact that $\pi((mp_1^{n_1},0,\dots,0))=mM$ for all $m\in\ZZ$.
\end{proof}

\begin{corollary}\label{pi-corollary}
Let $\tilde{A},\tilde{B}\subset\Lambda_M$, and let $A=\pi(\tilde A),B=\pi(\tilde B)$. Then $A\oplus B=\ZZ_M$ if and only if
$\tilde{A}\oplus \tilde{B}\oplus\calL_M=\ZZ^d$.
\end{corollary}

\begin{proof}
Assume that $\tilde{A}\oplus \tilde{B}\oplus\calL_M=\ZZ^d$. Since $\pi$ is one-to-one on $\Lambda_M$, we have $|A||B|=|\tilde{A}||\tilde B|=M$.
To prove that $A\oplus B=\ZZ_M$ is a tiling, it remains to verify that 
\begin{equation}\label{notildeAB}
\hbox{ if }a+b=a'+b'\hbox{ mod }M, \ a,a'\in A, \ b,b'\in B, \hbox{ then }(a,b)=(a',b').
\end{equation}
Let $a=\pi(\tilde a),a'=\pi(\tilde a'), b=\pi(\tilde b), b'=\pi(\tilde b')$ for some
$\tilde a,\tilde a' \in\tilde A$ and $\tilde b,\tilde b' \in \tilde B$. By (\ref{LM-equiv}), 
if $a+b\equiv a'+b'\hbox{ mod }M$, then $(\tilde a+\tilde b)-(\tilde a'+\tilde b')\in\mathcal{L}_M$. But since 
$\tilde{A}\oplus \tilde{B}\oplus\calL_M=\ZZ^d$, this can only happen when $\tilde a=\tilde a'$ and $\tilde b=\tilde b'$, so that $a=a'$ and $b=b'$. This proves (\ref{notildeAB}).

For the converse, we reverse the above argument. The details are left to the reader.
\end{proof}

\begin{remark}\label{pi-remark}
Let an integer tiling $A\oplus B=\ZZ_M$ be given. By Lemma \ref{pi-lemma}, we may represent $A,B$ as subsets of $\pi(\Lambda_M)$. Then there exist $\tilde{A},\tilde{B}\subset \Lambda_M$ such that $A=\pi(\tilde A),B=\pi(\tilde B)$, and Corollary \ref{pi-corollary} applies to these sets.
\end{remark}

This establishes a one-to-one correspondence between $M$-periodic tilings of $\ZZ$ and $\calL_M$-periodic tilings of $\ZZ^d$. Any tiling of $\ZZ$ by a finite set must be $M$-periodic for some $M$, therefore we may lift it to a multidimensional tiling as described above.
However, we caution the reader that there are many translational tilings of integer lattices by finite sets that cannot be matched to 1-dimensional integer tilings in this manner, either because they are non-periodic \cite{GT} or because their period lattice does not have the form required in (\ref{def-L}).


\subsection{Integer cube tilings}
\label{integer-cubes}

We now assume that $M=N^2$, where $N=p_1\dots p_d$ and $p_1,\dots,p_d$ are distinct primes. 
Let $M_i=M/p_i^2$ for $i=1,\dots,d$, and define $\pi:\ZZ^d\to\ZZ$ as in (\ref{def-pi}). 
Let $A\oplus B$ be an integer cube tiling as in Definition \ref{def-cubetilings}. Then the set $B$ in (\ref{defineB}) satisfies
$$
B=\pi (\Lambda_N).
$$
By Corollary \ref{pi-corollary} and Remark \ref{pi-remark}, there is a set $\tilde A\subset\Lambda_M$ such that
$\tilde A \oplus \Lambda_N \oplus\calL_M=\ZZ^d$. Let $T:=\tilde A \oplus\calL_M\subset\ZZ^d$, and let
\begin{equation}\label{basic rectangle}
R:= B+[0,1)^d = [0,p_1)\times \dots  \times [0,p_d) \subset\RR^d.
\end{equation}
Then $T\oplus R=\RR^d$ is a tiling of $\RR^d$ by translates of the rectangular box $R$.
Note the periodicity condition
\begin{equation}\label{periodic}
T \hbox{ is invariant under translations by all } t\in \calL_M=p_1^{2}\ZZ\times\dots\times p_d^{2}\ZZ.
\end{equation}
Conversely, given a tiling $T\oplus R=\RR^d$, where $R$ is the box in (\ref{basic rectangle}) and $T\subset\ZZ^d$ satisfies (\ref{periodic}), we can convert it to an $M$-periodic integer cube tiling by reversing the above procedure.

We can rescale $R$ to the unit cube $Q_d=[0,1)^d$ in $\RR^d$; this also rescales any tiling of $\RR^d$ by translates of $R$ to a tiling by translates of a unit cube (hence our terminology). However, for our purposes it will be easier to use the box $R$ as defined in (\ref{basic rectangle}), without rescaling, and rescale the unit cube instead when needed. This convention makes it easier to keep track of the additional restrictions on periodicity and translation vectors that our integer cube tilings must satisfy.

In this setting, (IKP2) is the direct analogue of (KP2). Indeed, for a given $i\in\{1,\dots,d\}$, we have $u+F_i \subset A$ for some $u\in\ZZ_M$ if and only if the tiling $T\oplus R=\RR^d$ contains a column in the direction of $e_i$ -- more precisely, the translation set $T$ contains a subset of the form $\{v+mp_ie_i:\ m\in\ZZ\}$ for some $v\in\ZZ^d$. This is (KP2) after rescaling $R$ to $Q_d$.

A little bit more care is needed in establishing the appropriate analogue of (KP1) in our context. In the tiling $T\oplus R=\RR^d$, two translates $t+R$ and $t'+R$ share a $(d-1)$-dimensional face if and only if $u=\pi(t)$ and $u'=\pi(t')$ satisfy
\begin{equation}\label{badcond}
u-u'\equiv \pm M/p_i \mod M
\end{equation}
for some $i\in\{1,\dots,d\}$.
However, (IKP1) is a more natural analogue of (KP1) for integer tilings than (\ref{badcond}), for the following reason. The coordinate system corresponding to the decomposition (\ref{CRT}) is not unique: for example, if $\pi$ is the projection defined in (\ref{def-pi}) and $r\in\ZZ$ is relatively prime to $M$, then the mapping $x\to r\pi(x)$ also establishes a linear bijection between $\Lambda_M$ (considered as a group mod $\calL_M$) and $\ZZ_M$. This leads to multiple representations of the same integer tiling $A\oplus B=\ZZ_M$ as $\calL_M$-periodic tilings of $\ZZ^d$. 
The property (IKP1) does not depend on the choice of such representation, while (\ref{badcond}) does depend on it.

As pointed out earlier, Theorem \ref{lowdim-thm-cube} is now a straightforward consequence of the existing results on Keller's conjecture for cube tilings.

\begin{proof}[Proof of Theorem \ref{lowdim-thm-cube}]
Let $A\oplus B=\ZZ_M$ be an integer cube tiling, and let $T\oplus R=\RR^d$ be the tiling of $\RR^d$ constructed in Section \ref{integer-cubes}. The theorem follows by applying the results of  \cite{LP1,LP2, BHMN} to a rescaling of $T$.
\end{proof}

In the other direction, a straightforward discretization of counterexamples to Keller's conjecture in dimensions 8 and higher \cite{LShor1}, \cite{Mackey} does not produce counterexamples to (IKP1) or (IKP2). This is because, in order to be able to convert a tiling of $\RR^d$ by the rectangular box $R$ back to an integer cube tiling, the translation set $T$ must satisfy the periodicity condition (\ref{periodic}). The unit cube tilings in \cite{LShor1}, \cite{Mackey} are all 2-periodic in each cardinal direction, and become $2p_i$-periodic in the $i$-th cardinal direction after rescaling the unit cube to $R$. Choose $i$ so that $p_i\neq 2$. If the tiling were also $p_i^2$-periodic in the $i$-th direction (as required in (\ref{periodic})), this would imply $p_i$-periodicity in that direction, hence a column parallel to $e_i$, contradicting the failure of (KP2).

One could ask whether the tilings in \cite{LShor1}, \cite{Mackey}, after rescaling, could be ``corrected'' to be $p_i^2$-periodic instead of $2p_i$-periodic in the $e_i$ direction for each $i$. This is in fact how we prove Theorem \ref{thm1}.

We note the following property of cube integer tilings.

\begin{lemma}\label{keller-sands}
Let $S\subset\Lambda_M$ with $|S|=N$, and define $B$ as in (\ref{defineB}). Then $S\oplus \Lambda_N \oplus \calL_M=\ZZ^d$ (hence
$\pi(S)\oplus B=\ZZ_M$ is a cube integer tiling) if and only if for every $a,a'\in S$ such that $a\neq a'$, there exists $i\in\{1,\dots,d\}$ such that $p_i\parallel a_i-a'_i$.
\end{lemma}

\begin{proof}
One proof of this is based on Sands's theorem. Let $A=\pi(S)$.
By Theorem \ref{thm-sands}, $A\oplus  B=\ZZ_M$ is a tiling
if and only if (\ref{e-sands}) holds. With $B$ defined in (\ref{defineB}), we have
$$
\Div( B)=\Big\{p_1^{\alpha_1}\dots p_d^{\alpha_d}:\ \alpha_i\in\{0,2\}\hbox{ for all }i\in\{1,\dots,d\} 
\Big\},
$$
so that (\ref{e-sands}) is equivalent to
$$
\Div(A)\subset \Big\{p_1^{\alpha_1}\dots p_d^{\alpha_d}|M:\ \alpha_i=1\hbox{ for some }i\in\{1,\dots,d\} 
\Big\}
$$
But this is equivalent to the condition in Lemma \ref{keller-sands}.

Alternatively, the lemma also follows from (a rescaled version of) Keller's theorem on cube tilings \cite{Keller}: if $T\oplus Q_d$ is a tiling of $\RR^d$, then for all $t,t'\in T$ with $t\neq t'$ there exists $i\in\{1,\dots,d\}$ such that $|t_i-t'_i|\in\NN$.
\end{proof}


\section{Cyclotomic divisibility tools}\label{cyclo-tools}

The notation below has been borrowed from \cite{LL1} and adapted to our setting.
We assume that $M=p_1^{n_1}\dots p_d^{n_d}$, where $p_1,\dots,p_d$ are distinct primes and $n_1,\dots,n_d\in\NN$.

We use $A(X)$, $B(X)$, etc. to denote polynomials modulo $X^M-1$ with integer coefficients. 
Each such polynomial $A(X)=\sum_{a\in\ZZ_M} w_A(a) X^a$ is associated with a weighted multiset in $\ZZ_M$, which we will also denote by $A$, with weights $w_A(x)$ assigned to each $x\in\ZZ_M$. (If the coefficient of $X^x$ in $A(X)$ is 0, we set $w_A(x)=0$.) In particular, if $A$ has $\{0,1\}$ coefficients, then
$w_A$ is the characteristic function of a set $A\subset \ZZ_M$. We use $\calm(\ZZ_M)$ to denote the 
family of all weighted multisets in $\ZZ_M$, and reserve the notation $A\subset \ZZ_M$ for sets.
We also use $\calm^+$ to denote the family of all weighted multisets in $\ZZ_M$, i.e.,
$$
\calm^+=\{A\in\calm(\ZZ_M): \ w_A(a)\geq 0\hbox{ for all }a\in\ZZ_M\}.
$$

For $D|M$, a {\em $D$-grid} in $\ZZ_M$ is a set of the form
$$
\Lambda(x,D):= x+D\ZZ_M=\{x'\in\ZZ_M:\ D|(x-x')\}
$$
for some $x\in\ZZ_M$. In particular, if $F_i$ is the fiber
\begin{equation}\label{def-Fi}
F_i=\{0,M/p_i,2M/p_i,\dots,(p_i-1)M/p_i\}
\end{equation}
for some $i$, we have $F_i=\Lambda(0,M/p_i)$. For a prime $p$ and an integer $m$, we write that $p^\alpha\parallel m$ if $p^\alpha|m$ but $p^{\alpha+1}\nmid m$.

\begin{definition}
\label{def-cuboids}
Let $M$ be as above. A \textbf{cuboid} is a multiset $\Delta \in \calm (\ZZ_M)$ associated to a mask polynomial of the form
\begin{equation}\label{def-delta}
\Delta(X)= X^c\prod_{j=1}^d \left(1-X^{r_jM/p_j}\right)
\end{equation}
 with $(r_j,p_j)=1$ for all $j$. Furthermore, if $A\in\calm(\ZZ_M)$, we define
\begin{equation}\label{delta-eval}
A[\Delta]=\sum_{x\in\ZZ_M} w_A(x)w_\Delta(x).
\end{equation}
\end{definition}

The geometric interpretation of $N$-cuboids $\Delta$ is as follows. With notation as in Definition \ref{def-cuboids}, let
\begin{equation}\label{def-DN}
D(M)=M/p_1\dots p_d.
\end{equation}
 Then the ``vertices'' $x\in\ZZ_M$ with $ w_\Delta(x)\neq 0$ form a full-dimensional rectangular box in the grid $\Lambda(c,D(M))$, with one vertex at $c$ and alternating $\pm 1$ weights. See Figure 1 for the geometric representation of a cuboid with $d=3$.

The following cyclotomic divisibility test is well known in the literature. The equivalence between (i) and (iii) is the Bruijn-R\'edei-Schoenberg theorem on the structure of vanishing sums of roots of unity (see \cite{deB}, \cite{LamLeung}, \cite{Mann}, \cite{Re1}, \cite{Re2}, \cite{schoen}). For the equivalence (i) $\Leftrightarrow$ (ii), see e.g.  \cite[Section 3]{Steinberger}, \cite[Section 3]{KMSV}, \cite[Section 5]{LL1}.

\begin{proposition}\label{cuboid}
Let $A\in\calm(\ZZ_M)$. Then the following are equivalent:

\medskip

(i) $\Phi_M(X)|A(X)$.

\medskip

(ii) For all cuboids $\Delta$, we have
\begin{equation}\label{id-3a}
A[\Delta]=0,
\end{equation}

\medskip

(iii) $A$ is a linear combination of fibers, so that
$$A(X)=\sum_{i} P_i(X)F_i(X) \mod X^M-1,$$
where $P_i(X)$ have integer (but not necessarily nonnegative) coefficients.
\end{proposition}

\begin{figure}
\centering
\includegraphics[scale=.2]{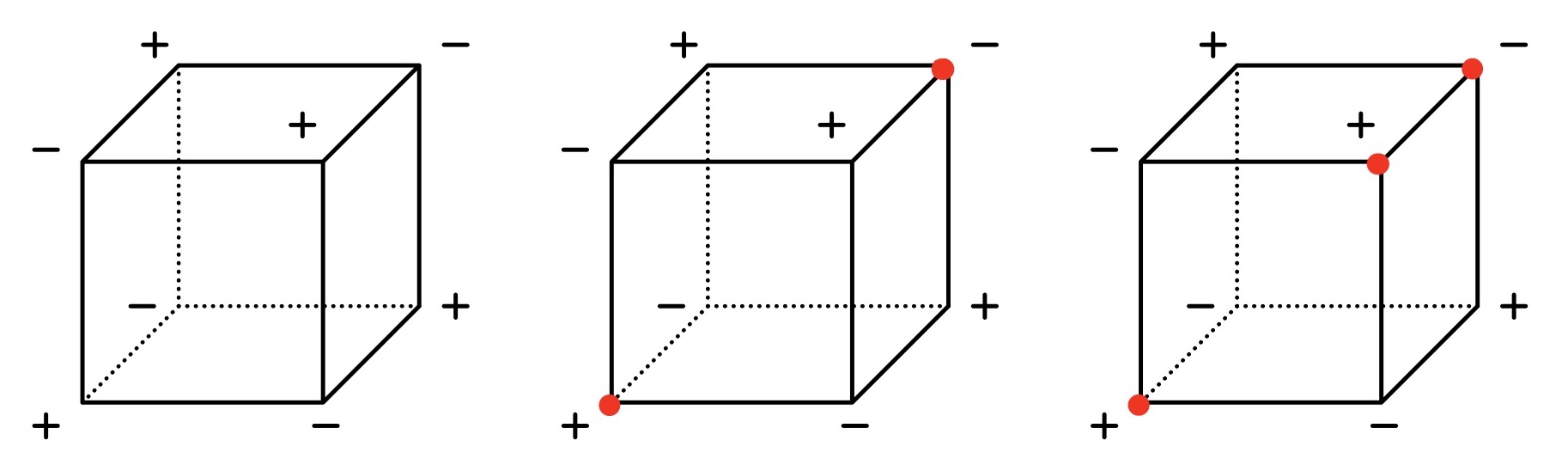}
\caption{Left:  A cuboid with $\pm 1$ weights labeled.  Middle and right:  Two cuboids, $\Delta$ and $\Delta'$, with points in $A$ labeled in red.  The middle cuboid is `balanced', as it satisfies $A[\Delta] = 0$.  The right cuboid is `unbalanced', with ${A}[\Delta'] > 0$.}
\end{figure}

Proposition \ref{cuboid} can be strengthened as follows if $M$ has only two distinct prime factors. This goes back to the work of de Bruijn \cite{deB}; see also \cite[Theorem 3.3]{LamLeung} for a self-contained proof, and \cite[Lemma 4.7]{LL2} for a statement in the language we use here.

\begin{lemma}\label{structure-thm} 
Let $A\in\calm^+(\ZZ_M)$.
Assume that $\Phi_{M}|A$, and that $M$ has at most two distinct prime factors $p_1,p_2$. Then $A$ is a linear combination of fibers with nonnegative weights. In other words,
$$A(X)=P_1(X)F_1(X)+ P_2(X)F_2(X) \mod X^M-1,$$
where $P_1,P_2$ are polynomials with nonnegative integer coefficients. If furthermore $M=p_1^{n_1}$ is a prime power, then the above holds with $P_2= 0$.

\end{lemma}

%


\section{Proof of Theorem \ref{thm-lowdim}}\label{sec-proof-thm-lowdim}

Let $M=p_1^{n_1}p_2^{n_2}p_3^{n_3}$ with $n_1,n_2,n_3\in\NN\cup\{0\}$, where $p_1,p_2,p_3$ are distinct primes. We need to prove the following:

\begin{itemize}

\item [(i)] (CKP) holds for $M$, and (IKP1) holds for any tiling $A\oplus B=\ZZ_M$.

\item [(ii)] If either $n_3=0$ or $\max(n_1,n_2,n_3)\leq 2$, then (IKP2) holds for any tiling $A\oplus B=\ZZ_M$.
\end{itemize}

Assume first that $n_3=0$. By Lemma \ref{structure-thm}, if $A\subset\ZZ_M$ satisfies $\Phi_M|A$, then $A$ is a union of non-overlapping fibers in the $p_1$ and $p_2$ directions. This implies (CKP) in this case. Furthermore, let $A\oplus B=\ZZ_M$ be a tiling. 
By (\ref{mask-e2}), $\Phi_M(X)$ must divide at least one of $A(X)$ and $B(X)$, hence (IKP2) follows as well.

Assume now that 
\begin{equation*}
\min(n_1,n_2,n_3)\geq 1,
\end{equation*}
and that $A\subset\ZZ_M$ obeys $\Phi_M|A$. We prove (i) in this case.
By translational invariance, we may assume that $0\in A$. Let $\Lambda:= \Lambda(0,D(M))$, then $A\cap\Lambda\neq\emptyset$. 
If $M/p_i\in\Div(A)$ for all $i\in\{1,2,3\}$, we are done. 
Suppose now that there exists $i\in\{1,2,3\}$ such that $M/p_i\not\in\Div(A)$. 
By \cite[Proposition 6.1]{LL2}, $A\cap\Lambda$ must contain either a fiber in some direction or a {\it diagonal boxes} configuration defined in \cite[Definition 5.1]{LL2}. It is easy to see that either the fiber or at least one of the diagonal boxes must contain a pair of points $a,a'\in A\cap\Lambda$ with $(a-a',M)=M/p_i$. This proves (i).

Finally, assume that 
\begin{equation*}
\min(n_1,n_2,n_3)\geq 1\hbox{ and }\max(n_1,n_2,n_3)\leq 2.
\end{equation*}
Let $A\oplus B=\ZZ_M$ be a tiling.
Assume without loss of generality that $0\in A$ and $\Phi_M(X)|A(X)$, and define $\Lambda$ as above. We will prove that $A\cap\Lambda$ contains a fiber.

By \cite[Proposition 5.2]{LL2}, $A\cap\Lambda$ must either be a union of disjoint fibers, or else it must contain a diagonal boxes configuration.
If $A\cap\Lambda$ is a union of disjoint fibers, we are done. It remains to consider the case when $A\cap\Lambda$ contains diagonal boxes. Then $A\cap\Lambda$ has one of the structures described in \cite[Section 7]{LL2} (if $M$ is odd) or \cite[Section 9]{LL3} (if $M$ is odd). In each of these cases, an intermediate step in the proofs of the Coven-Meyerowitz conditions in \cite{LL2,LL3} is proving that $A$ must then contain fibers that do not lie in $A\cap\Lambda$.


\section{Proof of Theorem \ref{cuboidKeller}}\label{sec-proof-thm-cuboidKeller}

Let $M=p_1^{n_1}\dots p_d^{n_d}$, where $p_1,\dots,p_n$ are distinct primes and $n_1,\dots,n_d \in\NN$.
Assume further that 
\begin{equation}\label{largeprimes2}
p_i>2^{i-2}\hbox{ for all }i\in\{1,\dots,d\}.
\end{equation}
We need to prove that
(CKP) holds for $M$.

Assume for contradiction that (CKP) fails for some $M$ and $A$ as above. Let $d$ be the lowest number of prime factors of $M$ for which this can happen; by Lemma \ref{structure-thm}, we must have $d\geq 3$. Let $ A\subset \ZZ_M$ be the hypothetical counterexample, so that $A$ is nonempty and we have 
$p_i>2^{i-2}$ for all $i\in\{1,\dots,d\}$, but $M/p_i\not\in\Div(A)$ for all $i$.

By translational invariance, we may assume that $0\in A$. Let $M_d:= M/p_d^{n_d}$ and
$$A_0:=\{p_d^{-n_d}a:\ a\in A \cap \Lambda(0, M/p_1\dots p_{d-1})\}.
$$
We cannot have $\Phi_{M_d}(X)|A_0(X)$, since otherwise $A_0\subset \ZZ_{M_d}$ would be a counterexample with $(d-1)$ prime factors and we assumed that $d$ is minimal. By Proposition \ref{cuboid}, there exists a cuboid
\begin{equation}\label{delta-e1}
\Delta_0(X)= X^c\prod_{j=1}^{d-1}  (1-X^{r_jM/p_j}),
\end{equation}
 with $p_d^{n_d}|c$ and $(r_j,p_j)=1$ for all $j$, such that $A_0[\Delta_0]\neq 0$. 
 Without loss of generality, we may assume that $A_0[\Delta_0]> 0$.

Let
$$
\Delta_s(X)= X^{c+s M/p_d} \prod_{j=1}^{d-1}  (1-X^{r_jM/p_j}),\ s=1,2,\dots,p_d-1,
$$
so that $(\Delta_0-\Delta_s)(X)$ is a full-dimensional cuboid in $\ZZ_M$ for each $s=1,2,\dots,d-1$. For visualization purposes, we suggest representing $\ZZ_M$ as a $d$-dimensional lattice as in Section \ref{subsec-convert}, with the $p_d$ direction vertical and 
$\Lambda(0, M/p_1\dots p_{d-1})$ represented as a lattice in a $(d-1)$-dimensional horizontal plane. Then $\Delta_s$ are partial cuboids stacked above the partial cuboid $\Delta_0$.

By Proposition \ref{cuboid}, we have $A[\Delta_0]-A[\Delta_s]=0$. But
$A[\Delta_0]=A_0[\Delta_0]> 0$, so that $A[\Delta_s]>0$. 
It follows that, for each $s$, at least one of those vertices of $\Delta_s$ that are counted with $+$ sign must belong to $A$. 

Let $v_1,v_2,\dots,v_\ell$ be the vertices of $\Delta_0$ that are counted with $+$ sign; we have $\ell=2^{d-2}$, half of the total number of vertices of $\Delta_0$. 
For each $s\in\{0,1,\dots,p_d-1\}$, any vertex $u$ of $\Delta_s$ that is counted with $+$ sign must satisfy $(M/p_d)|u-v_j$ for some $j$ (geometrically, $u$ must lie on the vertical line through one of the points $v_1,\dots,v_\ell$). We must have at least $p_d$ points of $A$ at such vertices, and, since we assume that $M/p_d\not\in\Div(A)$, at most one such point can lie on the same line. Therefore $p_d\leq 2^{d-2}$, contradicting our assumption.

\begin{remark}\label{cuboid-remark2}
Let $A\oplus B=\ZZ_M$ be an integer cube tiling. By (\ref{mask-e2}), $\Phi_M$ must divide at least one of $A(X)$ and $B(X)$. If we had 
$\Phi_M(X)|B(X)$, then $B$ would satisfy the condition (\ref{id-3a}) of Proposition \ref{cuboid} with $N=M$. However, this is false, since any $M$-cuboid with one vertex in $B$ cannot have any other vertices in $B$. Therefore $\Phi_M|A$, and Theorem \ref{cuboidKeller} implies that in this case we must in fact have $M/p_i\in\Div(A)$ for some $i$.
\end{remark}



\section{Proof of Theorem \ref{easyKeller}}\label{sec-proof-thm-easyKeller}


Theorem \ref{easyKeller} is a consequence of the following more general result. Let $M=p_1^{n_1}\dots p_d^{n_d}$. For $x\in\ZZ_M$ and $i\in\{1,\dots,d\}$, we define
$$
\Pi_i(x):=\{y\in\ZZ_M:\ p_i^{n_i}|(x-y)\}.
$$
In the Chinese Remainder Theorem geometric representation, $\Pi_i(x)$ is the $(d-1)$-dimensional hyperplane passing through $x$ and perpendicular to the $i$-th cardinal direction. Let also 
$$
m_A=\min_{a\in A}|A\cap \Pi_d(a)|,
$$
and similarly for $B$. We will continue to write $M_i=M/p_i^{n_i}$.

\begin{theorem}\label{easyKeller2}
Let $A\oplus  B=\ZZ_M$ be a tiling. If
\begin{equation}\label{pd-large2}
p_d>\max (m_A,m_B),
\end{equation}
then (IKP1) holds for this tiling, with $M/p_d\in \Div(A)\cup\Div(B)$.
\end{theorem}

We will rely on the following special case of \cite[Lemma 4.5]{LL3}. We include the short proof for completeness.

\begin{lemma}\label{splitting} (Splitting for fibers, \cite[Lemma 4.5]{LL3})
Assume that $A\oplus B=\ZZ_M$ is a tiling, and let $i\in\{1,\dots,d\}$. Let $z_k=k M/p_i$ for $k=0,1,\dots,p_i-1$. Let $a_k\in A,b_k\in B$ satisfy $a_k+b_k=z_k$. Then
one of the following must hold:
\begin{enumerate}
\item[(i)] We have $a_1,\dots,a_{p_i-1}\in\Pi_i(a_0)$ and $p_i^{n_i-1}\parallel b_k-b_{k'}$ for $k\neq k'$.
\item[(ii)] We have $b_1,\dots,b_{p_i-1}\in \Pi_i(b_0)$ and $p_i^{n_i-1}\parallel a_k-a_{k'}$ for $k\neq k'$.
\end{enumerate}

\end{lemma}

\begin{proof}
By translational invariance, we may assume that $a_0=b_0=0$. Fix $k\in\{1,\dots,p_i-1\}$. Then $a_k+b_k=M/p_i$, so that $(a_k,M_i)=(b_k,M_i)$. By divisor exclusion (\ref{e-sands}) we must have $(a_k,p_i^{n_i})\neq (b_k,p_i^{n_i})$, hence we are in one of the following two cases for each $k$:
\begin{itemize}
\item[(a)] $p_i^{n_i}|a_k$ and $p_i^{n_i-1}\parallel b_k$,
\item[(b)] $p_i^{n_i}|b_k$ and $p_i^{n_i-1}\parallel a_k$.
\end{itemize}
We now prove uniformity in $k$. Assume for contradiction that (a) holds for some $k$ and (b) holds for some $k'$. By the same argument as above, we have $M/p_i | (z_k-z_{k'})=(a_k-a_{k'})+(b_k-b_{k'})$, so that
$$
(a_k-a_{k'},M_i)=(b_k-b_{k'},M_i).
$$
But now we also have $(a_k-a_{k'},p_i^{n_i})=(b_k-b_{k'},p_i^{n_i})= p_i^{n_i-1}$. Hence $(a_k-a_{k'},M)=(b_k-b_{k'},M)$, contradicting (\ref{e-sands}). 

Assume now that (a) holds for all $k$. Then $a_1,\dots,a_{p_i-1}\in\Pi_i(0)$, and the second part of (i) follows from $a_k+b_k=kM/p_i$. Similarly, if (b) holds for all $k$, then (ii) follows.
\end{proof}

\begin{proof}[Proof of Theorem \ref{easyKeller2}]
Choose $a_0\in A$ and $b_0\in B$ so that $m_A=|A\cap \Pi_d(a_0)|$ and $m_B=|B\cap \Pi_d(b_0)|$. By translational invariance, we may assume that $a_0=b_0=0$. For $k=1,\dots, p_d-1$, let $a_k\in A,b_k\in B$ satisfy $a_k+b_k=kM/p_d$. Interchanging the sets $A$ and $B$ if necessary, we may further assume that the statement (i) of Lemma \ref{splitting} holds, so that $a_1,\dots,a_{p_d-1}\in\Pi_d(0)$. 

By (\ref{pd-large2}), at least two of the elements $a_0,a_1,\dots,a_{p_d-1}$ must coincide, so that $a_k=a_{k'}$ for some $k\neq k'$. But then $b_k-b_{k'}=kM/p_d - k'M/p_d = (k-k')M/p_d$, so that $M/p_d\in\Div(B)$ as claimed.
\end{proof}

Theorem \ref{easyKeller} now follows from Theorem \ref{easyKeller2} and Lemma \ref{plane-bound-lemma} below.

\begin{lemma}\label{plane-bound-lemma}(Plane bound; cf. \cite[Lemma 4.3]{LL2})
Let $A\oplus  B=\ZZ_M$ be a tiling. Assume that $|A|=m p_d^k$, where $m=(|A|,M_d)$ satisfies $(m,p_d)=1$. Then for any $x\in\ZZ_M$,
\begin{equation}\label{plane-bound}
|A\cap\Pi_d(x)|\leq m.
\end{equation}
A similar statement holds for $B$.
\end{lemma}

\begin{proof}
By \cite[Theorem B1]{CM}, there exist $0<s_1<s_2<\dots<s_k\leq n_d$ such that 
$$
\Phi_{p_d^{s_1}}(X)\dots  \Phi_{p_d^{s_k}}(X) \mid A(X).
$$
This implies that for each $j\in\{1,\dots,k\}$ and $x\in\ZZ_M$, we have
$$
|A\cap\Lambda(x,p_d^{s_j-1})|=p_d \, |A\cap\Lambda(x,p_d^{s_j})|.
$$
Iterating this, we get
$$
p_d^k\cdot  |A\cap\Pi_d(x)| \leq |A| = p_d^k m,
$$
and (\ref{plane-bound}) follows.
\end{proof}

In general, $m_A$ and $m_B$ could be significantly smaller than the upper bound in (\ref{plane-bound}): there could, for example, exist planes that contain only one element of $A$ or $B$.

If $A\oplus B=\ZZ_M$ is an integer cube tiling, the conclusion (ii) of Lemma \ref{splitting} is true for the simple geometric reason that a rectangular box has a fixed width in each direction. Moreover, for each $b\in B$ we have $|B\cap \Pi_d(b)|=p_1\dots p_{d-1}$, optimizing the bound in (\ref{plane-bound}). Theorem \ref{easyKeller2} then states that (IKP1) holds for integer cube tilings with 
\begin{equation}\label{pd-large3}
p_d>p_1\dots p_{d-1}.
\end{equation}Further improvements may be possible with additional assumptions. For example, we have the following.

\begin{lemma}\label{easyKeller3}
Let $A\oplus  B=\ZZ_M$ be an integer cube tiling, where $M=p_1^2\dots p_d^2$ and $p_1<p_2<\dots<p_d$ are distinct primes. If
\begin{equation}\label{pj-large}
p_j>p_2p_3\dots p_{j-1} \hbox{ for each }j\in\{3,\dots,d\},
\end{equation}
then (IKP1) holds for this tiling, with $M/p_j\in \Div(A)$ for some $j\in\{1,\dots,d\}$.
\end{lemma}

Unlike in (\ref{pd-large3}), we require (\ref{pj-large}) to hold for all $j\in\{3,\dots,d\}$; the payoff is that we lose the first prime factor from the right side of the inequality.

\begin{proof}
Assume towards contradiction that the lemma fails, and let $d$ be the smallest integer for which this happens. By Theorem \ref{lowdim-thm-cube}, we must have $d\geq 7$.
Translating $A$ if necessary, we may assume that $0\in A$. For each $x\in\ZZ_M$, we use $a(x)$ and $b(x)$ to denote the elements $a(x)\in A$ and $b(x)\in B$ that satisfy $a(x)+b(x)=x$.  

Let $\Lambda_0=\Lambda(0,M/p_1\dots p_{d-1})$. Suppose first that $p_d^2|a(x)$ for all $x\in\Lambda_0$. Let $A'=A\cap \Pi_d(0)$ and $B'=B\cap\Pi_d(0)$. Then $a(x)\in A'$ for all $x\in\Lambda_0$; furthermore, if $x,x'\in \Lambda_0$ are distinct, the sets $a(x)+B'$ and $a(x')+B'$ are disjoint. Hence $|A'|=p_1\dots p_d$, and $A'\oplus B'=\ZZ_{M/p_d^2}$ is an integer cube tiling that does not satisfy (IKP1), contradicting the minimality of $d$.

Assume now that there is $x'\in\Lambda_0$ such that $p_d^2\nmid a(x)$. Then there exist $j\in\{1,\dots,d-1\}$ and $y,y'\in \Lambda_0$ such that $(y-y',M)=M/p_j$ but $p_d^2\nmid a(y)-a(y')$. Translating $A$ again if necessary, we may assume that $y'=a(y')=0$.  

Let $\Lambda=\Lambda(0, M/p_jp_d)$ (a lattice in the 2-dimensional discrete plane passing through $0$, $y$, and $M/p_d$). By \cite[Lemma 4.7]{LL3} with $\alpha_i=\alpha_j=1$, at least one of the following holds:
\begin{equation}\label{plane-e1}
a(z)\in\Pi_d(0)\hbox{ for all }z\in\Lambda,
\end{equation}
\begin{equation}\label{plane-e2}
a(z)\in\Pi_j(0)\hbox{ for all }z\in\Lambda.
\end{equation}
But we have assumed that (\ref{plane-e1}) fails for $z=y$. Hence (\ref{plane-e2}) holds. 
Now the rest of the argument is as in Theorem \ref{easyKeller2}, but with the additional constraint that $p_j^2\mid a_i$ for all $i$.

Instead of invoking \cite[Lemma 4.7]{LL3}, we could have converted the integer cube tiling to a tiling of $\RR^d$ by a rectangular box (as we do in Section \ref{thm1-proof}), and then used the structure of tilings of a 2-dimensional plane by translates of a rectangle. This is equivalent to \cite[Lemma 4.7]{LL3} in this specific case, and yields the same conclusion.
\end{proof}


\section{Proof of Theorem \ref{thm1}}\label{thm1-proof}
\subsection{Setup} 
We will use the notation and conventions of Section \ref{integer-cubes}.
It suffices to prove that there exists a column-free tiling $T \oplus R = \R^d$ such that $T \subseteq \Z^d$ and $T$ is invariant under translations by elements of $\mathcal L_M$.  To construct this tiling, we will use an argument of Szab\'o \cite{Sz-Keller} to periodize the counterexample to (KP2) from \cite{Mackey}.  This procedure will not introduce any columns but could possibly introduce some shared faces.  The relevant result from \cite{Mackey} is the following:  There exists a tiling of $\R^8$ by unit cubes with centres in $\frac{1}{2}\Z^8$ such that no two cubes share an entire $7$-dimensional face.  Stacking this tiling, with appropriate half-integer offsets between adjacent layers, leads to an analogous tiling of $\R^d$ for every $d \geq 8$.  Thus, by rescaling, we have for each $d \geq 8$ a tiling $S \oplus R = \R^d$ with $S \subseteq \frac{p_1}{2}\Z \times \cdots \times \frac{p_d}{2}\Z$ that contains no shared faces.

We will need some notation and a lemma.  Given a set $S \subseteq \R^d$, define 
\begin{align*}
S_{i,a} := \{(s_1,\ldots,s_d) \in S \colon s_i = p_ia\} \quad\quad \text{for} \quad\quad i \in \{1,\ldots,d\},~a \in \frac{1}{2}\Z.
\end{align*}
If the projection of $S$ to the $i$th coordinate is contained in $\frac{p_i}{2}\Z$, then the sets $S_{i,a}$ (with $i$ fixed) form a partition of $S$.  We also define a shifted version of $S_{i,a}$, namely
\begin{align*}
S_{i,a}^* := \begin{cases}
S_{i,a} &\text{if } p_ia \in \Z,\\
S_{i,a}+\frac{1}{2}e_i &\text{if } p_ia \notin \Z.
\end{cases}
\end{align*}

\begin{lemma}\label{integer-difference-lem}
For $i \in \{1,\ldots,d\}$, let $\pi_i$ denote the projection $(x_1,\ldots,x_d) \mapsto (x_1,\ldots,\widehat{x_i},\ldots,x_d)$.  Suppose $S \oplus R = \R^d$ and that $s,s' \in S$ are such that $\pi_i(R+s)$ and $\pi_i(R+s')$ have nonempty intersection.  Then $s_i - s_i' \in p_i\Z$.
\end{lemma}
\begin{proof}
Fix a point $x \in \pi_i(R+s) \cap \pi_i(R+s')$, and consider the line $\ell = \pi_i^{-1}(x) \subset \R^d$.  The tiling $S \oplus R = \R^d$ restricts to a tiling of $\ell$ by segments of length $p_i$.  The endpoints of these segments are just the points where $\ell$ enters one box in the tiling and exits another.  Since $\ell$ intersects $R+s$ and $R+s'$ and is parallel to $e_i$, it follows that $s_i - s_i'$ is an integer multiple of $p_i$. 
\end{proof}

We now begin the construction of the tiling $T \oplus R = \R^d$ described above.  This will be done inductively.  We claim that for each $j \in \{0,\ldots,d\}$, there exists a tiling $S_j \oplus R = \R^d$ satisfying the following properties:
\begin{enumerate}
\item[(i)] If $j \geq 1$, then $S_j \subseteq \Z^j \times \frac{p_{j+1}}{2}\Z \times \cdots \times \frac{p_d}{2}\Z$.

\item[(ii)] If $1 \leq i \leq j$, then $S_j$ is $p_i^2 e_i$-periodic.

\item[(iii)] If $1 \leq i \leq j$, then $S_j$ contains no columns in the $e_i$ direction; i.e.~there is no $s \in S_j$ such that $\{s + np_ie_i \colon n \in \Z\} \subseteq S_j$.

\item[(iv)] If $j < i \leq d$, then $S_j$ contains no shared faces in the $e_i$ direction; i.e.~$p_ie_i \notin S_j - S_j$.
\end{enumerate}
Once this is proved, the desired tiling $T \oplus R = \R^d$ is obtained by taking $T = S_d$.

As mentioned above, there exists a tiling $S \oplus R = \R^d$ with $S \subseteq \frac{p_1}{2}\Z \times \cdots \times \frac{p_d}{2}\Z$ that contains no shared faces; we set $S_0 := S$.  Properties (i)--(iv) hold for $S_0$ (although (i)--(iii) are vacuous in this case).  Suppose $j \geq 1$ and that the claim holds with $j-1$ in place of $j$.  Let
\begin{align*}
S_j := \bigcup_{n \in \Z}\bigcup_{a \in \frac{1}{2}\Z \cap [0,p_j)} (S_{j-1})_{j,a}^* + np_j^2 e_j.
\end{align*}
We will show that $S_j \oplus R = \R^d$ and that this tiling satisfies properties (i)--(iv).

\subsection{Proof that $S_j \oplus R = \R^d$} We begin by showing that the translates $R+s$ with $s \in S_j$ cover $\R^d$.  Fix $x \in \R^d$.  Since $S_{j-1} \oplus R = \R^d$, there exists $s \in S_{j-1}$ such that $x \in R+s$.  We consider two cases.  Suppose first that $s_j \in p_j\Z$.  Let $n \in \Z$ be such that
\begin{align}\label{p_j^2 translation}
np_j^2 \leq s_j \leq (n+1)p_j^2 - p_j.
\end{align}
Set $x' := x - np_j^2e_j$ and let $s' \in S_{j-1}$ be such that $x' \in R+s'$.
Clearly $x \in R+s'+np_j^2e_j$; we claim that $s'+np_j^2e_j \in S_j$.  This would show that $x \in R+S_j$.  Using \eqref{p_j^2 translation} and the fact that $x' \in R+s'$ and $x \in R+s$, we find that
\begin{align*}
s_j' &> x_j' - p_j = x_j-np_j^2 - p_j \geq s_j - np_j^2 - p_j \geq -p_j
\intertext{and}
s_j' &\leq x_j' = x_j - np_j^2 < s_j+p_j - np_j^2 \leq p_j^2.
\end{align*}
Since $x$ and $x'$ differ only in the $j$th coordinate, Lemma \ref{integer-difference-lem} implies that $s_j - s_j' \in p_j\Z$.  We have assumed that $s_j \in p_j\Z$, and so we must have $s_j' \in p_j\Z$.  Therefore, the strict upper and lower bounds for $s_j'$ (displayed above) imply that $s_j' \in p_j\Z \cap [0, p_j^2)$.  Consequently,
\begin{align*}
s' \in \bigcup_{a \in \Z \cap [0,p_j)} (S_{j-1})_{j,a} = \bigcup_{a \in \Z \cap [0,p_j)} (S_{j-1})_{j,a}^*, 
\end{align*}
from which it easily follows that $s' + np_j^2e_j \in S_j$.  This concludes the proof in the case where $s_j \in p_j\Z$.

Next suppose that $s_j \notin p_j\Z$.  The projection of $S_{j-1}$ to the $j$th coordinate is contained in $\frac{p_j}{2}\Z$; if $j = 1$ then this follows from the definition of $S_0$, while if $j > 1$ then this is due to property (i).  Therefore, we have $s_j \in \frac{p_j}{2}\Z \setminus p_j\Z$.  Let $t \in S_{j-1}$ be such that $x-\delta e_j \in R+t$, where $\delta := 0$ if $p_j = 2$ and $\delta := \frac{1}{2}$ if $p_j > 2$.  By Lemma \ref{integer-difference-lem}, we have that $t_j - s_j \in p_j\Z$, and thus $t_j \in \frac{p_j}{2}\Z \setminus p_j\Z$. Let $n \in \Z$ be such that
\begin{align*}
np_j^2 + \frac{p_j}{2} \leq t_j \leq (n+1)p_j^2 - \frac{p_j}{2},
\end{align*}
and let $t' \in S_{j-1}$ be such that $x-(\delta+np_j^2)e_j \in R+t'$.  Clearly $x \in R + t' + (\delta+np_j^2)e_j$; we claim that $t' + (\delta+np_j^2)e_j \in S_j$.  Arguing as in the first case, we find that $$-\frac{p_j}{2} < t_j' < p_j^2 + \frac{p_j}{2}.$$  However, $t_j' \in \frac{p_j}{2}\Z \setminus p_j\Z$ (by another application of Lemma \ref{integer-difference-lem}), so we must have $t_j' \in (\frac{p_j}{2}\Z \setminus p_j\Z) \cap [0,p_j^2)$.  This means that
\begin{align*}
t' \in \bigcup_{a \in (\frac{1}{2}\Z \setminus \Z) \cap [0,p_j)} (S_{j-1})_{j,a} \quad\quad\text{so that}\quad\quad t'+\delta e_j \in \bigcup_{a \in (\frac{1}{2}\Z \setminus \Z) \cap [0,p_j)} (S_{j-1})_{j,a}^*.
\end{align*}
It easily follows that $t' + (\delta+np_j^2)e_j \in S_j$, completing the proof in the case where $s_j \notin p_j\Z$.

To prove that $S_j \oplus R = \R^d$, it remains to show that the translates of $R$ by elements of $S_j$ are pairwise disjoint.  Fix $s,s' \in S_j$, and let $a,a' \in \frac{1}{2}\Z \cap [0,p_j)$ and $n,n' \in \Z$ be such that $s \in (S_{j-1})_{j,a}^*+np_j^2e_j$ and $s' \in (S_{j-1})_{j,a'}^*+n'p_j^2e_j$.  We consider two cases.  Suppose first that $p_j(a-a') \notin \Z$.  Then $s_j - s_j' \notin p_j\Z$, so by Lemma \ref{integer-difference-lem}, the projections $\pi_j(R+s)$ and $\pi_j(R+s')$ are disjoint.  This implies that $R+s$ and $R+s'$ are disjoint.  Next suppose that $p_j(a-a') \in \Z$.  Then $|s_j-s_j'| = p_j|a-a'+p_j(n-n')|$.  Assuming that $R+s$ and $R+s'$ intersect, this difference is strictly less than $p_j$, forcing $a' = a$ and $n'=n$.  Consequently, we have $s-s' \in (S_{j-1})_{j,a}^* - (S_{j-1})_{j,a}^* \subseteq S_{j-1} - S_{j-1}$, and since $S_{j-1} \oplus R = \R^d$ is a tiling, this implies that $s = s'$.  This completes the proof that $S_j \oplus R = \R^d$.

\subsection{Proof that $S_j$ satisfies properties (i)--(iv)} Next, we verify that $S_j$ satisfies properties (i)--(iv), beginning with property (i).  Let $U$ be the projection of $S_j$ to the $j$th coordinate.  By construction, $S_j$ is a union of subsets of $S_{j-1}$ that have been translated in the $e_j$ direction only.  Therefore, if $j = 1$, then $S_j \subseteq U \times \frac{p_2}{2}\Z \times \cdots \times \frac{p_d}{2}\Z$ by the definition of $S_0$, while if $j > 1$, then property (i) applied to $S_{j-1}$ implies that $S_j \subseteq \Z^{j-1} \times U \times \frac{p_{j+1}}{2}\Z \times \cdots \times \frac{p_d}{2}\Z$.  Now, observe that the projection of $(S_{j-1})_{j,a}^*$ to the $j$th coordinate is always contained in $\Z$; this is because $p_ja \notin \Z$ implies $p_ja + \frac{1}{2} \in \Z$ for $a \in \frac{1}{2}\Z$.  It follows that $U \subseteq \Z$.

We now turn to property (ii).  It is clear that $S_j$ is $p_j^2e_j$-periodic.  Suppose that $1 \leq i \leq j-1$.  Property (ii) applies to $S_{j-1}$; thus $S_{j-1}$ is $p_i^2e_i$-periodic.  From this it follows that each $(S_{j-1})_{j,a}^*$ is $p_i^2e_i$-periodic, and consequently the same is true of $S_j$.

Next we verify property (iii).  Suppose that $S_j$ contains a column in the $e_i$ direction.  If $i < j$, then the $p_j^2e_j$-periodicity of $S_j$ implies that $\bigcup_{a \in \frac{1}{2}\Z \cap [0,p_j)}(S_{j-1})_{j,a}^*$ contains a column in the $e_i$ direction.  The projections of the sets $(S_{j-1})_{j,a}^*$ to the $j$th coordinate are distinct; therefore, the column must be contained in a single $(S_{j-1})_{j,a}^*$.  But $(S_{j-1})_{j,a}^*$ is a subset of $S_{j-1}$, possibly shifted in the $e_j$ direction, and $S_{j-1}$ does not contain a column in the $e_i$ direction, by property (iii).  If $i = j$, then the column intersects $\bigcup_{a \in \frac{1}{2}\Z \cap [0,p_j)}(S_{j-1})_{j,a}^*$ in $p_j \geq 2$ consecutive points;  these points are either all contained in $S_{j-1}$ or all contained in $S_{j-1}+\frac{1}{2}e_j$.  Either way, it follows that $p_je_j \in S_{j-1}-S_{j-1}$.  This contradicts the fact that $S_{j-1}$ has no shared faces in the $e_j$ direction, by property (iv).

Finally, we check property (iv).  Suppose that $S_j$ contains a shared face in the $e_i$ direction, for some $i > j$.  Then by $p_j^2e_j$-periodicity, the set $\bigcup_{a \in \frac{1}{2}\Z \cap [0,p_j)} (S_{j-1})_{j,a}^*$ also contains a shared face in the $e_i$ direction.  Since the sets $(S_{j-1})_{j,a}^*$ project to distinct values in the $j$th coordinate, this shared face must occur within a single $(S_{j-1})_{j,a}^*$.  But this implies that $S_{j-1}$ contains a shared face in the $e_i$ direction, contradicting property (iv).  This completes the proof of the claim; therefore the proof of the theorem is complete.


\section{Integer Keller properties and the Coven-Meyerowitz conjecture}\label{sec-CM}

We conclude this paper by providing more details on the Coven-Meyerowitz conjecture and its relationship to our results here. Given a finite set $A\subset\NN\cup\{0\}$, 
let $S_A$ be the set of prime powers
$p^\alpha$ such that $\Phi_{p^\alpha}(X)$ divides $A(X)$. 
Coven and Meyerowitz \cite{CM} proposed the following tiling conditions: 

\smallskip
{\it (T1) $A(1)=\prod_{s\in S_A}\Phi_s(1)$,}

\smallskip
{\it (T2) if $s_1,\dots,s_k\in S_A$ are powers of distinct
primes, then $\Phi_{s_1\dots s_k}(X)$ divides $A(X)$.}

\noindent
The main results of \cite{CM} are:

\begin{itemize}

\item if $A$ satisfies (T1), (T2), then $A$ tiles $\ZZ$ by translations;

\item  if $A$ tiles $\ZZ$ by translations, then (T1) holds;

\item if $A$ tiles $\ZZ$ by translations, and if $|A|$ has at most two distinct prime factors,
then (T2) holds.
\end{itemize}

The statement that (T2) must hold for all finite tiles of $\ZZ$ has become known as the {\it Coven-Meyerowitz conjecture}. 
At the time of this writing, (T2) is known to hold for both tiles in any tiling of period $M=M'p_1\dots p_k$, where:
\begin{itemize}
\item $p_1,\dots,p_k$ are distinct primes not dividing $M'$,
\item either $M'$ has at most two distinct prime factors, or else $M'=p^2q^2r^2$, where $p,q,r$ are distinct primes.
\end{itemize}
The case when $M'$ has at most two distinct prime factors can be resolved by the methods developed in \cite{CM}. The statement above follows \cite[Corollary 6.2]{LL1}; see also \cite[Theorem 1.5]{M}, \cite[Proposition 4.1]{shi}, and the comments on \cite{Tao-blog}.
The proof combines the main results of \cite{CM} with the ``subgroup' reduction'', developed in \cite{CM} and stated in \cite[Theorem]{LL1} in the form needed here. 
More recently, {\L}aba and Londner \cite{LL2,LL3} proved that (T2) holds for both tiles in any tiling of period $M=p^2q^2r^2$, where $p,q,r$ are distinct primes. 
Together with the subgroup reduction, this also implies the above statement with 
$M'=p^2q^2r^2$.

We note an interpretation of (T2) in terms of tiling complements.  Coven and Meyerowitz \cite{CM} proved that if $A$ satisfies (T1) and (T2), then it admits a tiling of the form $A\oplus B^\flat =\ZZ_M$, where $M={\rm lcm}(S_A)$ and $B^\flat$ is an explicitly constructed and highly structured ``standard" tiling complement depending only on the prime power cyclotomic divisors of $A(X)$. Conversely, if a tile $A$ admits a tiling with a standard tiling complement, it satisfies (T2) \cite[Proposition 3.4]{LL1}.

As a special case, the set $B$ defined in (\ref{defineB}) is the standard tiling set with 
$$\Phi_{p_1}(X)\dots \Phi_{p_N}(X)| B(X).$$
It follows that both sets $A$ and $B$ in any integer cube tiling, no matter how badly behaved, must satisfy the Coven-Meyerowitz conditions. 
At the same time, examples such as those in Theorem \ref{thm1} show that the structure of integer tilings may be quite complicated even when (T2) is known in advance. 
They also shed light on the viability of certain approaches to proving (T2), and specifically, of the approach initiated in \cite{LL2,LL3}.

A high-level overview of the proof of (T2) in \cite{LL2,LL3} is as follows. Let $M=p^2q^2r^2$, where $p,q,r$ are distinct primes. Assume that $A\oplus B=\ZZ_M$ is a tiling, with $|A|=|B|=pqr$. Without loss of generality, we may assume that $\Phi_M(X)|A(X)$. Let $\Lambda:=pqr\ZZ_M$, and consider the sets $A_a:=A\cap(\Lambda+a)$ with $a\in A$. The proof now splits in two cases. If each set $A_a$ is a union of disjoint fibers in some direction (depending on $a$, but the same for the entire set $A_a$), this fibering property is used to split up the tiling into tilings of smaller groups. Assume now that there is an $a\in A$ such that $A_a$ is not a union of disjoint fibers in a fixed direction. In this ``unfibered" case, the authors are able to find fibers in $A$ that do not lie in $\Lambda+a$, then shift these fibers so that the original tiling $A\oplus B=\ZZ_M$ is replaced by a new tiling $A'\oplus B=\ZZ_M$ with a simpler structure. The procedure continues until the entire set $A$ is replaced by the subgroup coset $a+\Lambda$; at that point, (T2) for both $A$ and $B$ follows from \cite[Proposition 3.4]{LL1}.

Some of the methods and intermediate results of \cite{LL1,LL2,LL3} extend to integer tilings with more prime factors and/or more scales. However, Theorem \ref{thm1} shows that any approach to proving (T2) that is based on fiber shifting cannot lead to a full resolution of the conjecture for tilings with 8 or more prime factors. This also answers Question 2 in \cite[Section 9]{LL1} in the negative.

We also mention the possibility of proving partial results concerning either the structure or certain specific elements of the divisor sets. For example, Conjecture 9.1 in \cite{LL1} states the following:

\begin{quote}
Let $A\oplus B=\ZZ_M$ be a tiling. Let $p$ be a prime such that $p^n\parallel M$ and $\Phi_{p^n}|A$. Then $M/p\not\in\Div(B)$.
\end{quote}

This would follow from (T2), but might also be proved independently as a weaker result. By Theorem \ref{thm-sands}, if $A\oplus B=\ZZ_M$ then $\Div(A)\cap\Div(B)=\{M\}$. In the specific case of integer cube tilings, we always have $\Phi_{p^2}|A$ for each $p|M$. 
If there exist integer cube tilings that do not satisfy (IKP1), this would not disprove the above conjecture, but it would prove the weaker statement that $M/p\not\in\Div( A)$ is possible.


\section{Acknowledgement}

The first author was supported by NSERC Discovery Grants 22R80520 and GR010263.
The second author 
was supported by NSERC Discovery Grant 22R80520.



\bibliographystyle{amsplain}

\begin{thebibliography}{99}


\bibitem{BHMN} J. Brakensiek, M. Heule, J. Mackey, D. Narv\'aez, {\it The Resolution of Keller's Conjecture},
 in: Peltier, N., Sofronie-Stokkermans, V. (eds.), {\it Automated Reasoning: 10th International Joint Conference, IJCAR 2020, Paris, France, July 1–4, 2020, Proceedings, Part I,} Lecture Notes in Computer Science, vol. 12166, Springer, pp. 48–65.

\bibitem{CM} E. Coven, A. Meyerowitz, {\it Tiling the integers
with translates of one finite set}, J. Algebra 212 (1999),
161--174.

\bibitem{deB} N.G. de Bruijn, {\it On the factorization of cyclic groups}, Indag. Math. 15 (1953), 370--377.

\bibitem{GT} R. Greenfeld, T. Tao. {\it A counterexample to the periodic tiling conjecture}, Ann. Math., to appear.

\bibitem{Keller} O.H. Keller, {\it Uber die l\"uckenlose Erf\"ullung des Raumes mit W\"urfeln}, J. Reine Angew. Math. 163 (1930), 231--248.

\bibitem{KMSV} G. Kiss, R. D. Malikiosis, G. Somlai, M. Vizer, 
{\it On the discrete Fuglede and Pompeiu problems}, Analysis \& PDE 13 (2020), 765-788.

\bibitem{Kol} M. N. Kolountzakis, {\it Translational tilings of the integers with long periods}, Electronic J. Comb.10(1) (2003), \# R22.

\bibitem{LL1} I. {\L}aba, I. Londner, {\it Combinatorial and harmonic-analytic methods for integer tilings},
Forum of Mathematics - Pi (2022), Vol. 10:e8, 46 pp.

\bibitem{LL2} I. {\L}aba, I. Londner, {\it The Coven-Meyerowitz tiling conditions for 3 odd prime factors},
Invent. Math. 232(1) (2023), 365-470.


\bibitem{LL3} I. {\L}aba, I. Londner, {\it Splitting for integer tilings and the Coven-Meyerowitz tiling conditions}, 
preprint, 2022.

\bibitem{LShor1} J. Lagarias, P. Shor, {\it Keller’s cube-tiling conjecture is false in high dimensions},
Bull. Amer. Math. Soc. 27 (1992), 279--283.


\bibitem{LS} J.C. Lagarias, S. Szab\'o, {\it Universal spectra and Tijdeman's
conjecture on factorization of cyclic groups}, J. Fourier Anal. Appl. 1(7) (2001), 63--70. 
%
%
%
%
\bibitem{LamLeung} T.Y. Lam and K.H. Leung, {\it On vanishing sums of roots of unity},
J. Algebra 224 (2000), 91--109.
%

\bibitem{LP1} M. {\L}ysakowska, K. Przes{\l}awski: {\it On the structure of cube tilings and unextendible systems of cubes in low dimensions}, European Journal of Combinatorics, 32 (8) (2011), 1417--1427.

\bibitem{LP2} M. {\L}ysakowska, K. Przes{\l}awski: {\it Keller's conjecture on the existence of columns in cube tilings of $\mathbb{R}^{n}$}, Advances in Geometry 12 (2) (2012), 329--352.

\bibitem{Mackey} J. Mackey, {\it A Cube Tiling of Dimension Eight with No Facesharing}, 
Discrete Comput. Geom. 28 (2002), 275–279.


%


\bibitem{Mann} H. B. Mann, {\it On Linear Relations Between Roots of Unity}, Mathematika 12, Issue 2 (1965), 107--117.

\bibitem{New} D.J. Newman, {\it Tesselation of integers}, J. Number Theory 9 (1977), 107--111.


\bibitem{Perron1} O. Perron, {\it Uber l\"uckenlose Ausf\"ullung des $n$-dimensionalen Raumes durch kongruente W\"urfel}, Math. Z. 46 (1940), 1--26.

\bibitem{Perron2} O. Perron, {\it Uber l\"uckenlose Ausf\"ullung des $n$-dimensionalen Raumes durch kongruente W\"urfel. II}, Math. Z. 46 (1940), 161--180.


\bibitem{PR}
B. Poonen and M. Rubinstein: {\it Number of Intersection Points Made by the Diagonals of a Regular Polygon}, SIAM J. Disc. Math. 11 (1998), 135--156.




\bibitem{Re1} L. R\'edei, {\it \"Uber das Kreisteilungspolynom}, Acta Math. Hungar. 5 (1954), 27--28.

\bibitem{Re2} L. R\'edei, {\it Nat\"urliche Basen des Kreisteilungsk\"orpers}, Abh. Math. Sem. Univ. Hamburg 23 (1959), 180--200.


\bibitem{Sands} A. Sands, {\it On Keller's conjecture for certain cyclic
groups}, Proc. Edinburgh Math. Soc. 2 (1979), 17--21.


\bibitem{Sands-Hajos} A. Sands, {\it On a conjecture of G. Haj\'os}, Glasgow Math. J. 15 (1974), 88-89.


\bibitem{schoen}
I. J. Schoenberg, {\it A note on the cyclotomic polynomial}, Mathematika 11 (1964), 131-136.

\bibitem{shi}
R. Shi, {\it Fuglede's conjecture holds on cyclic groups $\ZZ_{pqr}$
}, Discrete Analysis 2019:14, 14pp.


\bibitem{Steinberger} J. P. Steinberger, {\it Minimal vanishing sums of roots of unity with large coefficients},
Proc. London Math. Soc. 97 (3) (2008), 689--717.


\bibitem{Steinberger2} J. P. Steinberger, {\it Tilings of the integers can have superpolynomial periods}, Combinatorica 29(4) (2009), 503--509.


\bibitem{Sz} S. Szab\'o, {\it A type of factorization of finite abelian groups},
Discrete Math. 54 (1985), 121--124.

\bibitem{Sz-Keller}
S.~Szab\'o, \emph{A reduction of Keller's conjecture}, Period.~Math.~Hungar., 17 (1986), no.~4, 265--277.



%



\bibitem{Tao-blog} T. Tao, {\it Some notes on the Coven-Meyerowitz conjecture}, blog post available at https://terrytao.wordpress.com/2011/11/19/some-notes-on-the-coven-meyerowitz-conjecture/





\end{thebibliography}

\vfill

\noindent{\sc Department of Mathematics, UBC, Vancouver,
B.C. V6T 1Z2, Canada}

\noindent{\it benjamin.b.bruce@gmail.com, ilaba@math.ubc.ca}

\end{document}